\documentclass[preprint,12pt]{elsarticle}
\usepackage[latin1]{inputenc}
\usepackage{amsmath, amssymb, amsthm}

\journal{Stochastic Processes and their Applications}

\newcommand{\cov}{\operatorname{Cov}}
\newcommand{\var}{\operatorname{Var}}

\newcommand{\card}{\mathop{\rm Card}}

\newtheorem{theo}{Theorem}[section]
\newtheorem{pr}{Proposition}[section]

\newtheorem{lem}{Lemma}[section]

\newtheorem{rmk}{Remark}[section]

\begin{document}

\begin{frontmatter}



\title{Random pinning model with finite range correlations: disorder relevant regime}


\author{Julien Poisat}

\address{Institut Camille Jordan,
Universit\'e Claude Bernard Lyon 1,
43 boulevard du 11 novembre 1918
69622 Villeurbanne cedex
France, poisat@math.univ-lyon1.fr}

\begin{abstract}
The purpose of this paper is to show how one can extend some results on disorder relevance obtained for the random pinning model with i.i.d disorder to the model with finite range correlated disorder. In a previous work, the annealed critical curve of the latter model was computed, and equality of quenched and annealed critical points, as well as exponents, was proved under some conditions on the return exponent of the interarrival times. Here we complete this work by looking at the disorder relevant regime, where annealed and quenched critical points differ. All these results show that the Harris criterion, which was proved to be correct in the i.i.d case, remains valid in our setup. We strongly use Markov renewal constructions that were introduced in the solving of the annealed model.
\end{abstract}

\end{frontmatter}

\section{Introduction}

Let $\tau = (\tau_n)_{n\geq0}$ be a recurrent renewal process starting at $\tau_0 = 0$ with interarrival distribution
 \begin{equation}\label{defK}K(n) = P(\tau_1 = n) = L(n)n^{-(1+\alpha)}\end{equation}
 with $L(\cdot)$ a slowly varying function. By recurrent we mean that $\sum_{n\geq1} K(n)=1$. The interarrival times, or stretches, are the random variables $T_k = \tau_k -\tau_{k-1}$, $k\geq 1$. For all $n\geq 0$, $\delta_n$ will denote the indicator function of the event $\{n\in\tau\}$ and sometimes we will use the notation \begin{equation}\label{imath}\imath_N = \sum_{n=1}^N \delta_n.\end{equation} Independently of $\tau$, let $\omega = (\omega_n)_{n\in\mathbb{Z}}$ be a Gaussian process with $0$ mean and variance $1$. We assume that there exists an integer $q\geq 1$ such that $\rho_n = 0$ as soon as $n>q$, where $\rho_n = \cov(\omega_0,\omega_n)$ (finite range correlations assumption). Its law will be denoted by $\mathbb{P}$. The Hamiltonian of the system at size $n\geq 1$, parameters $(\beta,h)$ in $(\mathbb{R}^+,\mathbb{R})$ and pinning potential $\omega$ is $$H_n = \sum_{k=1}^n (\beta\omega_k + h)\delta_k.$$ The corresponding (quenched) partition function is the quantity $$Z_{n,\beta,h,\omega} = E(\exp(H_n)\delta_n)$$ and the annealed partition function is $$Z^a_{n,\beta,h} = \mathbb{E}Z_{n,\beta,h,\omega}.$$ The (infinite volume) quenched and annealed free energy functions are defined respectively as $$F(\beta,h) = \lim_{n\rightarrow+\infty} (1/n) \log Z_{n,\beta,h,\omega}\geq 0$$ (in the almost sure and $L^1(\mathbb{P})$ sense) and $$F^a(\beta,h) = \lim_{n\rightarrow+\infty} (1/n) \log Z^a_{n,\beta,h}\geq 0.$$ For the existence of these limits, we refer to \cite{Poisat_frc}. The localized (resp. delocalized) phase is the region of parameters for which the quenched free energy is positive (resp. null). Both phases are separated from each other by a concave critical curve $h_c(\beta) = \sup\{h\in\mathbb{R} : F(\beta,h) = 0\}$. If one defines the annealed critical curve as $h_c^a(\beta) = \sup\{h\in\mathbb{R} : F^a(\beta,h) = 0\}$, then the following inequality holds: $h_c(\beta) \geq h_c^a(\beta)$. Disorder will be said relevant if the previous inequality is strict, and irrelevant otherwise. 
 
 The case $q=0$, which is the case of i.i.d. disorder, is the most studied one. In this setup, the annealed model reduces to the homogeneous model (the $\beta =0$ case), which is fully solvable (see \cite{GG_Book}), so all annealed features are known. In particular, $h_c^a(\beta) = -\beta^2/2$. A lot has been done lately on the issue of disorder relevance/irrelevance. For $\alpha= 0$, disorder is always irrelevant (see \cite{Alexander_loop_exp_one,Cheliotis_DenH}). If $\alpha\in(0,\frac{1}{2})$ or $\alpha = \frac{1}{2}$ and $\sum_{n\geq1} \frac{1}{nL(n)^2} < \infty$ then there exists a critical value $\beta_c >0$ such that disorder is irrelevant for $\beta \leq \beta_c$ (and in this case quenched and annealed critical exponents are the same) and relevant otherwise (see \cite{ Alexander_Sido, MR2561435, Fabio_replica, Cheliotis_DenH,  lacoin_irrelevance}). If $\alpha>\frac{1}{2}$ then disorder is always relevant (i.e $\beta_c = 0$) and we know the order of the difference between quenched and annealed critical curves for small $\beta$ (see \cite{Fabio_replica, Alexander_Sido,  Alexander_quenched, Derrida_al_relevance}). All these results have proved that the value $\alpha = \frac{1}{2}$ is critical regarding disorder relevance, a fact that corresponds in physics literature to the Harris criterion. The controversial case $\alpha = \frac{1}{2}$, with $L(\cdot)$ not subject to the previous condition, is probably the most delicate. For this we refer to the works \cite{Derrida_al_relevance,  GG_T_L_Marginal, MR2779401} and references therein. We also mention the recent work \cite{2011arXiv1104.4969B} where the quenched critical point and exponent are given for a particular environment (based on a renewal sequence) with long-range correlations.

Part of the theory has been extended to the case $q\geq1$ in \cite{Poisat_frc}, where the motivation is to study the effect of disorder correlations on the model. More precisely, the following has been proved:
\begin{theo}\label{annealed}
For all $\beta\geq 0$, $h_c^a(\beta) = -\frac{\beta^2}{2} - \log \lambda(\beta)$ where $\lambda(\cdot)$ is defined in (\ref{lambda}). Moreover, $h_c^a(\beta) \stackrel{\beta\searrow 0}{\sim} -\frac{\beta^2}{2}\left(1 + 2\sum_{n=1}^q \rho_n P(n\in\tau) \right)$.
\end{theo}
\begin{theo}\label{irrelevance}
If $\alpha\in(0,\frac{1}{2})$ or $\alpha = \frac{1}{2}$ and $\sum_{n\geq1} \frac{1}{nL(n)^2} < \infty$ then there exists $\beta_c>0$ such that for all $\beta\leq\beta_c$, $h_c(\beta) = h_c^a(\beta)$ and $\lim_{\epsilon \searrow 0} \frac{\log F(\beta,h_c(\beta)+\epsilon)}{\log\epsilon} = \frac{1}{\alpha}$.
\end{theo}
Theorem \ref{annealed} shows that correlations can modify the critical curves in a quantitative way, even at the leading order in $\beta$ whereas Theorem \ref{irrelevance} suggests that the Harris criterion remains valid. Moreover, it is shown in \cite[Proposition 5.1]{Poisat_frc} that the annealed critical exponent remains the same as in the homogeneous case (the proof is done for $0\leq\alpha\leq 1/2$ but it is straightforward to adapt it to $\alpha> 1/2$) More precisely we have
\begin{pr}\label{ann_expo}
 For all $\beta\geq0$, there exists a slowly varying function $L_{\beta}$ such that
 \begin{equation*}
  F^a(\beta,h_c^a(\beta)+\Delta) \stackrel{\Delta\searrow0}{\sim} L_{\beta}(1/\Delta)\Delta^{\max(1,1/\alpha)}.
  \end{equation*}
\end{pr}
The idea in \cite{Poisat_frc} is to exhibit a Markov renewal structure to solve the annealed model (by solve we mean find critical points and exponents). The purpose of this paper is to show how one can also use this construction to generalize to our case the results of disorder relevance obtained in the case of i.i.d disorder. This complements our study of the model with finite range correlated disorder.

\section{Results}

The following results were first obtained in the case of i.i.d disorder. We show that they also hold in our case.
\begin{theo}\label{relevance1}
Let $\alpha\in(0,\frac{1}{2})$. There exists $\beta_0<\infty$ such that for all $\beta>\beta_0$, $h_c(\beta) > h_c^a(\beta)$.
\end{theo}
Theorem \ref{relevance1} is also true for $\alpha>\frac{1}{2}$, but in this case we have stronger results:
\begin{theo}\label{relevance3}
 Let $\alpha\in(\frac{1}{2},1)$. For all $\beta>0$, $h_c(\beta) > h_c^a(\beta)$ .
\end{theo}
\begin{theo}\label{relevance2}
Let $\alpha>1$. There exists $a>0$ such that for all $\beta\leq1$, $h_c(\beta) \geq h_c^a(\beta) + a\beta^2$. Moreover, $h_c(\beta)>h_c^a(\beta)$ for all $\beta>0$.
\end{theo}

To put it simply, we need to adapt the proofs to the world of Markov renewal processes. At some places, the fact that the underlying Markov renewal law at the annealed critical point depends on $\beta$ requires extra work. Theorem \ref{relevance1}, which was proved in \cite{1157.60090} in the i.i.d case relies on a fractional moment estimate technique. In a few words, this consists in bounding from above fractional moments of the quenched partition functions to prove that the free energy is null for some values of the parameters. In the i.i.d case, an explicit value of $\beta_0$ can be given: $\beta_0 = \inf\{\beta\geq0 : \frac{\beta^2}{2} - h(K)>0\} = \sqrt{2h(K)}$, where $h(K) = -\sum_{n\geq1}K(n)\log K(n)$ is the entropy of $K(\cdot)$. Our value of $\beta_0$ is not explicit, but it can still be implicitly defined as the first $\beta$ for which the difference between an energetic term and an entropic term becomes positive (see proof of Theorem \ref{relevance1}). Theorems \ref{relevance2} and \ref{relevance3} were proved in \cite{Derrida_al_relevance} in the case of i.i.d. disorder. The result given there is more complete when $\alpha \in (1/2,1)$, since it was proved in this case that for all $\epsilon>0$, there exists $a(\epsilon)>0$ such that
\begin{equation}\label{shift_h_c}
h_c(\beta)\geq h_c^a(\beta) + a(\epsilon)\beta^{\frac{2\alpha}{2\alpha-1}+\epsilon}
\end{equation}
for all $\beta\leq 1$ (see Remark \ref{rmk2}).
\begin{rmk}
 We claim that these results should hold without the Gaussian assumption (as in the i.i.d. setup of \cite{Derrida_al_relevance}), namely for disorder sequences of form
 $\omega_n = a_0\varepsilon_n + a_1\varepsilon_{n-1} + \ldots + a_q \varepsilon_{n-q}$, for all $n\geq0$ and a positive integer $q$, where the $a_i$'s are real numbers and $(\varepsilon_n)_{n\in\mathbb{Z}}$ is a sequence of i.i.d. random variables with finite exponential moments. One can check that this kind of disorder still satisfies the finite range correlations assumption and allows us to compute the exponential moments (with respect to disorder)  of $\sum_{k=1}^n \omega_k \delta_k$ (by rewriting it as a sum of the $\varepsilon_k$'s). Therefore, the methods used in \cite{Poisat_frc} should be adaptable, as well as the proofs of this paper. However, the expressions that one gets can be quite involved so we choose not to develop this point here.
\end{rmk}

\section{Definitions and notations}

We need to remind some definitions and notations from \cite{Poisat_frc}. They are necessary for the Markov renewal construction mentionned beforehand, and we will need them in our proofs. First we define the following mapping: $t\in\mathbb{N}^* \mapsto t^*\in E:= \{1,\ldots,q,\star\}$ with $t^* = t$ if $1\leq t \leq q$ and $t^* = \star$ otherwise. Loosely speaking, $\star$ is an abstract state refering to interarrival times greater than $q$, and it obeys the following rule: for all $z\in E$, $z+\star = \star + z = \star$. Vectors of $q$ consecutive interarrival times (resp. elements of $E$) will usually be denoted by $\overline{t} = (t_1, \ldots, t_q)$ (resp $x=(x_1,\ldots,x_q)$). If $(t_n)_{n\geq1}$ is a sequence of interarrival times, then we use the notation $\overline{t}_n = (t_n, \ldots, t_{n+q-1})$ and $\overline{t}^*_n = (t^*_n, \ldots, t^*_{n+q-1})$. We also define the consistency condition: $\overline{s} \rightsquigarrow \overline{t}$ (resp. $\overline{s}^* \rightsquigarrow \overline{t}^*$) if for all $i\in\{2,\ldots,q\}$, $s_i = t_{i-1}$ (resp. $s_i^* = t_{i-1}^*$).

A function $G$ is defined on $(\mathbb{N}^*)^q$ by $G(\overline{t}) = \rho_{t_1} + \rho_{t_1+t_2} + \ldots + \rho_{t_1+\ldots + t_q}$, but since $G(\overline{s}) = G(\overline{t})$ as soon as $\overline{s}^* = \overline{t}^*$, we can as well define it on $E^q$ by
\begin{equation*}
 G(x) = \rho_{x_1} + \rho_{x_1+x_2} + \ldots + \rho_{x_1+\ldots +x_q}
\end{equation*}
if we agree that $\rho_{\star} = 0$. This reduction to a finite state space is helpful for the resolution of the annealed model. Indeed, we can make the following transfer matrix appear
\begin{equation*}
Q^*_{\beta}(x,y) = e^{\beta^2 G(y)}K(y_q) \mathbf{1}_{\{x \rightsquigarrow y\}}.
\end{equation*}
where $K(\star):= \sum_{n>q} K(n)$, and we define 
\begin{equation}\label{lambda}
\lambda(\beta) = \hbox{ Perron-Frobenius eigenvalue of } Q^*_{\beta}
\end{equation}
which is the quantity appearing in Theorem \ref{annealed}. To solve the annealed model, a law $P_{\beta}$ is introduced in \cite{Poisat_frc}, that we recall here. Let $r^*_{\beta} = (r^*_{\beta}(x))_{x\in E^q}$ be a positive right eigenvector of $Q^*_{\beta}$ associated to $\lambda(\beta)$. Define for all $\overline{t}$ in $(\mathbb{N}^*)^q$, $r_{\beta}(\overline{t}) = r^*_{\beta}(\overline{t}^*)$ and $Q_{\beta}$ the infinite matrix $Q_{\beta}(\overline{s},\overline{t})= e^{\beta^2 G(\overline{t})}K(t_q)\mathbf{1}_{\{\overline{s}\rightsquigarrow\overline{t}\}}$. Then the matrices $\tilde{Q}_{\beta}$ and $\tilde{Q}^*_{\beta}$ respectively defined by \begin{equation*}\tilde{Q}_{\beta}(\overline{s},\overline{t}) := \frac{Q_{\beta}(\overline{s},\overline{t})r_{\beta}(\overline{t})}{\lambda(\beta)r_{\beta}(\overline{s})}\end{equation*} and \begin{equation} \label{tildeQstar}\tilde{Q}^*_{\beta}(\overline{s}^*,\overline{t}^*) := \frac{Q_{\beta}^*(\overline{s}^*,\overline{t}^*)r_{\beta}^*(\overline{t}^*)}{\lambda(\beta)r_{\beta}^*(\overline{s}^*)}\end{equation} are Markov transition kernels (resp. on $(\mathbb{N}^*)^q$ and $E^q$), see \cite[Lemma 4.1]{Poisat_frc}. The law $P_{\beta}$ is then defined on the interarrival times $(T_n)_{n\geq1}$ by $$P_{\beta}(T_1 = t_1, \ldots, T_q = t_q) = \prod_{k=1}^q K(t_k)$$ and for all $k\geq0$, $$P_{\beta}(T_{k+q+1} = t_{q+1} | T_{k+1} = t_1, \ldots, T_{k+q} = t_q) = \tilde{Q}_{\beta}(\overline{t}_1,\overline{t}_2).$$ Then one remarks (\cite[Section 4.4]{Poisat_frc}) that under $P_{\beta}$, $(\tau_n)_{n\geq0}$ is a delayed Markov renewal process with  modulating Markov chain $(\overline{T}_{k-q}^*)_{k\geq q+1}$, and with the following semi-Markov kernel: for all $n\geq1$, $x,y\in E^q$,
$$P_{\beta}(T_{k+q+1}=n, \overline{T}_{k+2}^*=y | \overline{T}_{k+1}^*=x) = \tilde{Q}_{\beta}^*(x,y) \frac{K(n)}{K(y_q)}\mathbf{1}_{\{n^* = y_q\}}.$$ We define $E_{\beta}$ as the expectation with respect to $P_{\beta}$.

\section{Proof of Theorem \ref{relevance1}}

In this section we prove Theorem \ref{relevance1}. For all $\gamma$ in $(\frac{1}{1+\alpha},1]$, for all $x$,$y$ in $E^q$ we define
\begin{equation*}
 \hat{Q}^*_{\beta,\gamma}(x,y) = \left\{ \begin{array}{ccc} \frac{K(y_q)^{\gamma}}{\lambda(\beta)^{\gamma}} \exp\left\{\frac{\beta^2}{2}\gamma(\gamma-1) + \gamma^2\beta^2 G(y) \right\} \mathbf{1}_{\{x \rightsquigarrow y\}}& \hbox{if} & y_q \neq \star,\\ \frac{\sum_{n>q}K(n)^{\gamma}}{\lambda(\beta)^{\gamma}} \exp\left\{\frac{\beta^2}{2}\gamma(\gamma-1) + \gamma^2\beta^2 G(y) \right\} \mathbf{1}_{\{x \rightsquigarrow y\}} & \hbox{if} & y_q = \star .
                          
                         \end{array} \right.
\end{equation*}
The condition $\alpha>0$ ensures that $(\frac{1}{1+\alpha},1]$ is nonempty. We denote by $\Lambda(\beta,\gamma)$ the Perron-Frobenius eigenvalue of $\hat{Q}_{\beta,\gamma}^*$. We will use the following lemma:
\begin{lem}\label{lemma_1}
 If $\Lambda(\beta,\gamma)<1$ then there exists $\delta>0$ such that
 \begin{equation*}
  \lim _{N\rightarrow+\infty} \frac{1}{N}\log \mathbb{E}Z_{N,\beta,h_c^a(\beta)+\delta}^{\gamma} = 0.
 \end{equation*}
\end{lem}

\begin{proof}[Proof of Lemma \ref{lemma_1}]
 We start by decomposing the partition function. For every real number $h$, we have
\begin{equation*}
Z_{N,\beta,h} = \sum_{n=1}^N \sum_{\substack{t_1,\ldots,t_n \geq 1\\ t_1+\ldots+t_n=N}}\exp(\sum_{i=1}^n (\beta\omega_{t_1+\ldots+t_i}+h))\prod_{i=1}^n K(t_i).
\end{equation*}
For all $\gamma\in(0,1)$ and nonnegative $(a_i)_{1\leq i\leq n}$, we have 
 \begin{equation}\label{fractrick}
 (a_1 +\ldots + a_n)^{\gamma}\leq a_1^{\gamma}+\ldots +a_n^{\gamma},
 \end{equation} hence
 \begin{equation}\label{rev6_1}
   Z_{N,\beta,h}^{\gamma} \leq \sum_{n=1}^N \sum_{\substack{t_1,\ldots,t_n \geq 1\\ t_1+\ldots+t_n=N}}\exp(\gamma \sum_{i=1}^n (\beta\omega_{t_1+\ldots+t_i}+h))\prod_{i=1}^n K(t_i)^{\gamma}.
 \end{equation}
 From our assumptions on the correlations of $\omega$,
 \begin{align*}
 \var\left( \sum_{i=1}^n \omega_{t_1+\ldots +t_i}\right) &= n + 2 \sum_{i=2}^n (\rho_{t_i} + \rho_{t_i+t_{i+1}} + \ldots + \rho_{t_i + \ldots + t_n}) \\
 &\leq n + 2 \sum_{i=1}^n G(\overline{t}_i) + c
 \end{align*}
 where $c=(q+1)(|\rho_1|+\ldots+|\rho_q|)$, and for every possible value for $t_{n+1},\ldots,t_{n+q}$. The inequality comes by bounding from above boundary effects ($i=1$ and $n-q\leq i \leq n$). Since $\omega$ is Gaussian, we have
 \begin{equation}\label{rev6_2}
 \mathbb{E}\left( \exp\left(\beta\gamma \sum_{i=1}^n \omega_{t_1+\ldots +t_i} \right) \right)\leq C(\beta) \exp\left( \frac{\gamma^2 \beta^2 n}{2} + \gamma^2 \beta^2 \sum_{i=1}^n G(\overline{t}_i) \right).
 \end{equation}
 where $C(\beta)$ is a constant. From (\ref{rev6_1}) and (\ref{rev6_2}), by choosing $$h = h_c^a(\beta) + \delta = -\frac{\beta^2}{2} - \log \lambda(\beta) + \delta,$$ we get
\begin{equation}\label{decompZgamma}
 \mathbb{E}(Z_{N,\beta,h_c^a(\beta)+\delta}^{\gamma}) \leq C(\beta,\delta) \sum_{n=1}^N \sum_{\substack{\overline{t}_1,\ldots,\overline{t}_n \\ t_1 +\ldots+t_n=N}} K(t_1)^{\gamma}\ldots K(t_q)^{\gamma}\prod_{k=1}^{n-1} e^{\delta}\hat{Q}_{\beta,\gamma}(\overline{t}_i,\overline{t}_{i+1})
 \end{equation}
 where 
 \begin{equation*}
  \hat{Q}_{\beta,\gamma}(\overline{s},\overline{t}) :=  \frac{K(t_q)^{\gamma}}{\lambda(\beta)^{\gamma}} \exp\left\{\frac{\beta^2}{2}\gamma(\gamma-1) + \gamma^2\beta^2 G(\overline{t}) \right\} \mathbf{1}_{\{\overline{s} \rightsquigarrow \overline{t}\}}.
 \end{equation*}
Now we take $\gamma$ in $(\frac{1}{1+\alpha},1]$. Let $\beta$ and $\gamma$ be such that $\Lambda(\beta,\gamma)<1$. Let $r^*$  be a positive right eigenvector of $\hat{Q}^*_{\beta,\gamma}$, associated to $\Lambda(\beta,\gamma)$. Define  $r$ on $(\mathbb{N}^*)^q$ by $r(\overline{s}) = r^*(\overline{s}^*)$. Then one can observe that $$\sum_{\overline{t}\in(\mathbb{N}^*)^q} \hat{Q}_{\beta,\gamma}(\overline{s},\overline{t})r(\overline{t}) = \Lambda(\beta,\gamma) r(\overline{s}).$$
As a consequence, for all $\overline{s}$ and for $\delta>0$ small enough, \begin{equation*}\sum_{\overline{t}} e^{\delta} \hat{Q}_{\beta,\gamma}(\overline{s},\overline{t})\frac{r(\overline{t})}{r(\overline{s})} = e^{\delta}\Lambda(\beta,\gamma)  \leq 1.\end{equation*}
This allows us to define a process with the following kernel: for all $k\geq0$
 \begin{equation}\label{defPhat}
  \hat{P}(T_{k+q+1}=t_{q+1} | T_{k+1}= t_1, \ldots, T_{k+q} = t_q) = \hat{Q}_{\beta,\gamma}(\overline{t}_1,\overline{t}_2)\frac{r(\overline{t}_2)}{r(\overline{t}_1)}e^{\delta}
 \end{equation}
 and the possibly positive probability
 \begin{equation*}
  \hat{P}(T_{k+q+1}=+\infty | \overline{T}_k = \overline{s}) = 1 -  e^{\delta}\Lambda(\beta,\gamma),
 \end{equation*}
 with the initial conditions $\hat{P}(T_1 = t_1,\ldots ,T_q = t_q ) = \frac{1}{c(\gamma)^q}K(t_1)^{\gamma}\ldots K(t_q)^{\gamma}$ where $c(\gamma) = \sum_{n\geq 1}K(n)^{\gamma}$. Notice that (\ref{defPhat}) tells how to sample an interarrival time conditionally to the past, only if previous interarrival times are finite. As soon as an interarrival time is infinite, all coming interarrival times coming after are defined as $+\infty$. Therefore, we may write \begin{equation*}\mathbb{E}(Z_{N,\beta,h_c^a(\beta)+\delta}^{\gamma})  \leq c(\gamma)^q C(\beta,\delta) \max_{x,y\in E^q}\{r^*(y)/r^*(x)\} \hat{P}(N\in\tau)\end{equation*} and since $\mathbb{E}(Z_{N,\beta,h_c^a(\beta)+\delta}^{\gamma}) \geq C' K(N)^{\gamma}$ (by restricting the partition function to the event $\{\tau_1=N\}$), we get the result.
\end{proof}

\begin{proof}[Proof of Theorem \ref{relevance1}]
 Suppose that $\beta$ is such that there exists $\gamma$ in $(\frac{1}{1+\alpha},1)$ satisfying the condition of Lemma \ref{lemma_1}. Then, for $\delta>0$ small enough,
 \begin{align*}
  \frac{1}{n}\mathbb{E}\log Z_{n,\beta,h_c^a(\beta)+\delta} = \frac{1}{\gamma n} \mathbb{E}\log Z_{n,\beta,h_c^a(\beta)+\delta}^{\gamma} &\stackrel{\hbox{(Jensen)}}{\leq} \frac{1}{\gamma n} \log \mathbb{E} Z_{n,\beta,h_c^a(\beta)+\delta}^{\gamma}\\ &\stackrel{\hbox{(Lemma \ref{lemma_1})}}{\longrightarrow_{n\rightarrow+\infty}} 0,
 \end{align*}
which implies $F(\beta, h_c^a(\beta)+\delta)=0$, that is $h_c(\beta) > h_c^a(\beta)$. Therefore, it is sufficient to prove, since $\Lambda(\beta,1)=1$, that for $\beta$ large enough,
\begin{equation*}
 \partial_{\gamma}\Lambda(\beta,\gamma)\arrowvert_{\gamma=1^-} > 0.
\end{equation*}
The first step is to compute $\partial_{\gamma}\hat{Q}^*_{\beta,\gamma}\arrowvert_{\gamma=1^-}$. Straightforward computations yield (we write $\hat{Q}^*_{\beta}$ instead of $\hat{Q}^*_{\beta,1}$)
\begin{equation}\label{derivative1}
 \partial_{\gamma}\hat{Q}^*_{\beta,\gamma}\arrowvert_{\gamma=1^-}(x,y) =  \left(\frac{\beta^2}{2} - \log\lambda(\beta) + 2\beta^2G(y) + \log K(y_q) \right) \hat{Q}^*_{\beta}(x,y)
\end{equation}
if $y_q \neq \star$; and if $y_q = \star$,
\begin{align}\label{derivative2}
  \partial_{\gamma}\hat{Q}^*_{\beta,\gamma}\arrowvert_{\gamma=1^-}(x,y) =  &\left(\frac{\beta^2}{2} - \log\lambda(\beta) + 2\beta^2G(y) + \log K(\star) \right) \hat{Q}^*_{\beta}(x,y) \\&+ \left(\sum_{n>q} \frac{K(n)}{K(\star)} \log \frac{K(n)}{K(\star)}\right) \hat{Q}^*_{\beta}(x,y).
\end{align}
Let us denote by $l^*_{\beta}$ (resp. $r^*_{\beta}$) a left row (resp. right column) eigenvector of $\hat{Q}^*_{\beta}$ associated to 1, with positive coordinates, normalized such that
\begin{equation*}
 l^*_{\beta}\cdot r^*_{\beta} = 1.
\end{equation*}
Then (see \cite[A.8]{GG_Book} for instance) we have
\begin{equation}\label{derivative3}
 \partial_{\gamma}\Lambda(\beta,\gamma)_{\arrowvert_{\gamma=1^-}} = l^*_{\beta}\cdot \partial_{\gamma}\hat{Q}^*_{\beta,\gamma}\arrowvert_{\gamma=1^-}r^*_{\beta}.
\end{equation}
We denote by $\pi_{\beta}$ the probability on $E^q$ defined by 
\begin{equation*}
 \pi_{\beta}(x) = l^*_{\beta}(x) r^*_{\beta}(x).
\end{equation*}
This probability is in fact the invariant probability of the Markov chain on $E^q$ with transition kernel $\tilde{Q}_{\beta}^*$ defined in (\ref{tildeQstar}). Indeed, for all $y$,
\begin{align*}
\sum_{x} \pi_{\beta}(x)\tilde{Q}_{\beta}^*(x,y) = \sum_{x} l_{\beta}^*(x)r_{\beta}^*(x)\tilde{Q}^*_{\beta}(x,y) &= \sum_{x} l^*_{\beta}(x)\hat{Q}_{\beta}^*(x,y)r^*_{\beta}(y) \\&= l^*_{\beta}(y)r^*_{\beta}(y) \\&=\pi_{\beta}(y).
\end{align*}
In the sequel, $X^{(n)} = (X^{(n)}_1,\ldots, X^{(n)}_q)$, $n\geq0$, will refer to a Markov chain on $E^q$ with kernel $\tilde{Q}^*_{\beta}$ and initial law $\pi_{\beta}$. Its law will be denoted by $P_{\pi_{\beta}}$ and $E_{\pi_{\beta}}$ will be the expectation with respect to $P_{\pi_{\beta}}$. Putting (\ref{derivative1}) and (\ref{derivative2}) in (\ref{derivative3}), and using stationarity of $\pi_{\beta}$, we get
\begin{align}\label{sum2}
 \partial_{\gamma}\Lambda(\beta,\gamma)_{\arrowvert_{\gamma=1^-}} &= \frac{\beta^2}{2} -\log\lambda(\beta) + 2\beta^2 E_{\pi_{\beta}}(G(X^{(0)})) + E_{\pi_{\beta}}(\log K(X^{(0)}_q))\nonumber \\&+ P_{\pi_{\beta}}(X^{(0)}_q = \star) \left(\sum_{n>q} \frac{K(n)}{K(\star)} \log \frac{K(n)}{K(\star)}\right)
\end{align}
Analyzing the behaviour of $\lambda(\beta)$ and $\pi_{\beta}$ for large values of $\beta$ is not a trivial task, because it depends on the maxima of the function $G$ (see for instance \cite{MR1659185} and references therein on this topic). We will rather transform the last expression so that the proof does not rely on the large $\beta$ analysis  of these quantities. The sum in (\ref{sum2}) can be reinterpreted as a sum of energy and entropy terms : the term $\sum_{n>q} \frac{K(n)}{K(\star)} \log \frac{K(n)}{K(\star)}$ is the opposite of the entropy of the kernel $K_q(n) := \frac{K(n)}{K(\star)}\mathbf{1}_{\{n>q\}}$, we denote by $h(K_q)$. The specific entropy $h(\tilde{Q}^*_{\beta})$ of the stationary Markov chain $(X^{(n)})_{n\geq0}$ (see \cite[pp.59-63]{ShieldsBook}) can be rewritten as, using (\ref{tildeQstar}),
\begin{align*}
 -h(\tilde{Q}^*_{\beta}) &\stackrel{\hbox{(def)}}{=} E_{\pi_{\beta}}(\log \tilde{Q}^*_{\beta}(X^{(0)},X^{(1)}))\\
 & = \beta^2 E_{\pi_{\beta}}(G(X^{(1)})) + E_{\pi_{\beta}}(\log K(X^{(1)}_q))\\& - \log \lambda(\beta) + E_{\pi_{\beta}}(\log r^*_{\beta}(X^{(1)})) -E_{\pi_{\beta}}(\log r^*_{\beta}(X^{(0)}))\\  &\stackrel{\hbox{(stationarity)}}{=} \beta^2 E_{\pi_{\beta}}(G(X)) + E_{\pi_{\beta}}(\log K(X_q)) - \log \lambda(\beta).
\end{align*}
Therefore, we may write:
\begin{equation}\label{EQ}
 \partial_{\gamma}\Lambda(\beta,\gamma)_{\arrowvert_{\gamma=1^-}} = \frac{\beta^2}{2} + \beta^2 E_{\pi_{\beta}}(G(X)) - h(\tilde{Q}^*_{\beta}) - h(K_q)P_{\pi_{\beta}}(X_q = \star).
\end{equation}
Note that the entropy $h(K_q)$ is finite because $\alpha>0$. As the specific entropy of a process on the finite state space $E^q$, for all $\beta$, $h(\tilde{Q}^*_{\beta})$ is nonnegative and bounded above by $\log \card(E^q)$, so the last two terms of (\ref{EQ}) are bounded. We are now going to conclude the proof by showing that 
\begin{equation*}
 \frac{\beta^2}{2} + \beta^2 E_{\pi_{\beta}}(G(X)) \stackrel{\beta\rightarrow+\infty}{\longrightarrow} +\infty.
\end{equation*}
Let $h(\tilde{Q}^*_{\beta}|\tilde{Q}^*_0)$ be the specific relative entropy (see \cite{deuschel1989large} for instance) of the stationary Markov chain with transition matrix $\tilde{Q}^*_{\beta}$ with respect to the one with transition matrix $\tilde{Q}^*_0$, defined as the limit of $(1/n) h(\tilde{Q}^*_{\beta}\arrowvert_{\mathcal{F}_n}|\tilde{Q}^*_0\arrowvert_{\mathcal{F}_n})$, where $\mathcal{F}_n$ is the $\sigma$-algebra generated by the random variables $X^{(k)}$ for $0\leq k \leq n$. We have
\begin{align*}
 h(\tilde{Q}^*_{\beta}\arrowvert_{\mathcal{F}_n}|\tilde{Q}^*_0\arrowvert_{\mathcal{F}_n}) &= E_{\pi_{\beta}}\left( \log \left(\frac{\pi_{\beta}(X^{(0)}) \prod_{i=1}^n \tilde{Q}^*_{\beta}(X^{(i-1)},X^{(i)}}{\pi_0(X^{(0)}) \prod_{i=1}^n \tilde{Q}^*_0(X^{(i-1)},X^{(i)})} \right) \right)\\
 &= \beta^2\sum_{i=1}^n E_{\pi_{\beta}}(G(X^{(i)})) - n \log\lambda(\beta) + E_{\pi_{\beta}}(\log r^*_{\beta}(X^{(n)})) \\& - E_{\pi_{\beta}}(\log r^*_{\beta}(X^{(0)})) + E_{\pi_{\beta}}\left(\log\left(\frac{\pi_{\beta}(X^{(0)})}{\pi_0(X^{(0)})}\right)\right)\\
 & \stackrel{(\hbox{stationarity})}{=}  \beta^2\left(E_{\pi_{\beta}}(G(X^{(0)})) - \log\lambda(\beta)\right)n + h(\pi_{\beta}|\pi_{0})
\end{align*}
and so
\begin{equation*}
 h(\tilde{Q}^*_{\beta}|\tilde{Q}^*_0) = \beta^2 E_{\pi_{\beta}}(G(X^{(0)})) - \log\lambda(\beta),
\end{equation*}
which is a nonnegative quantity. Thus,
\begin{equation*}
 \frac{\beta^2}{2} + \beta^2 E_{\pi_{\beta}}(G(X)) \geq \frac{\beta^2}{2} + \log\lambda(\beta) = -h_c^a(\beta).
\end{equation*}
Since $h_c^a(\beta) \stackrel{\beta\rightarrow+\infty}{\longrightarrow} -\infty$ (because $h_c(0)=0$, $h^a_c(\beta)<0$ for some $\beta>0$ and it is concave in $\beta$), the proof is complete.
\end{proof}

\section{Proofs of Theorems \ref{relevance3} and \ref{relevance2}}

In this section we shall prove Theorem \ref{relevance3} and Theorem \ref{relevance2}. In a first part, we adapt the fractional moment technique developed in \cite{Derrida_al_relevance} to our case. It is a refinement of the fractional moment technique of the previous section from which we show that the free energy is null if a certain sum $\varrho$ depending on $\beta$,$h$,$\gamma$ and a scale $k$ is small (see Lemma \ref{lemma_rho} and (\ref{sum})). The way we make this quantity small depends whether $\alpha$ is greater than $1$ or between $1/2$ and $1$. In the sequel, the functions $L_i(\cdot)$ will refer to slowly varying functions.

\subsection{Fractional moments}

In the following, we take this definition of the Hamiltonian:
\[H_j(\beta,h,\omega,\tau) = \sum_{k=0}^{j-1} (\beta\omega_k + h)\delta_k,\]
which does not change the value of the limit free energy. We recursively define the following subset of $\tau$: $\hat{\tau}_0 = 0$ and for all $n\geq0$
\begin{equation*}
\hat{\tau}_{n+1} = \inf\{\tau_k>\hat{\tau}_n : \tau_k - \tau_{k-1} > q\},
\end{equation*}
i.e $\hat{\tau}$ is the subset of renewal points that come just after a stretch strictly larger than $q$. Let us also define the following partition functions:
\begin{align*}
 \hat{Z}_{j,\beta,h,\omega} &:= E\left( \exp(H_j(\beta,h,\omega)) \mathbf{1}_{\{\hat{\tau}\cap\{1\ldots j\} = \{j\}\}}\right), \\
  \check{Z}_{j,\beta,h,\omega} &:= E\left(\exp(H_j(\beta,h,\omega)) \delta_j \mathbf{1}_{\{\hat{\tau}\cap\{1\ldots j\}=\emptyset\}}\right),\\
   \tilde{Z}_{j,\beta,h,\omega} &:= E\left(\exp(H_j(\beta,h,\omega)) \mathbf{1}_{\{j\in\hat{\tau}\}}\right).
\end{align*}
In other words, $\hat{Z}_j$ is the partition function restricted to the event ``$j$ is a renewal point and the only stretch strictly larger than $q$ is the one just before $j$'', $\check{Z}_j$ the restriction to the event ``$j$ is a renewal point and all stretches before it are smaller than $q$'', and $\tilde{Z}_j$ the restriction to ``$j$ is a renewal point, the stretch just before it is strictly larger $q$''. 

Let $k$ be an integer that we shall specify later. We decompose $\tilde{Z_n}$ the following way: $l$ is the last element of $\hat{\tau}$ strictly before $k$ and $r$ is the first element of $\hat{\tau}$ greater or equal to $k$. This yields, by Markov property:
\begin{equation}\label{dec}
 \tilde{Z}_{n,\omega} = \sum_{0\leq l <k} \sum_{r=k}^n \tilde{Z}_{l,\omega} \hat{Z}_{r-l,\theta^l\omega} \tilde{Z}_{n-r,\theta^r\omega}, 
\end{equation}
if we agree that $\hat{Z}_j = \tilde{Z}_j = 0$ if $1\leq j \leq q$, and $\hat{Z}_0 = \tilde{Z}_0 = 1$. Observe that the three factors in the sum, seen as disorder functions, are independent because of the finite range assumption and our construction of $\hat{\tau}$. From (\ref{dec}) and (\ref{fractrick}) we deduce that for all $\gamma\in(0,1)$,
\begin{equation}\label{ineq1}
 \tilde{Z}_{n,\omega}^{\gamma} \leq \sum_{0\leq l <k} \sum_{r=k}^n \tilde{Z}_{l,\omega}^{\gamma} \hat{Z}_{r-l,\theta^l\omega}^{\gamma} \tilde{Z}_{n-r,\theta^r\omega}^{\gamma}
\end{equation}
and if we define the sequence $A_n = \mathbb{E}\tilde{Z}_{n,\omega}^{\gamma}$, we have by independence, for $n\geq k$,
\begin{equation}\label{ren}
 A_n \leq \sum_{l=0}^{k-1}\sum_{r=k}^n A_l \mathbb{E}(\hat{Z}_{r-l}^{\gamma}) A_{n-r}.
\end{equation}
Let
\begin{equation*}
 \hat{K}(j) = \hat{K}(j,\beta,h,\gamma) =\left\{ \begin{array}{ccc}
                                                  \mathbb{E}(\hat{Z}_{j,\beta,h,\omega}^{\gamma}) & \hbox{if} & j>q\\
                                                  0 & \hbox{if} & j\leq q.
                                                 \end{array} \right.
\end{equation*}
We have the following lemma:
\begin{lem}\label{lemma_rho}
 If $\beta$ and $h$ are such that there exists $k\geq1$ and $\gamma\in(0,1)$ for which
 \begin{equation}\label{cdt}
  \varrho(\beta,h,\gamma,k) := \sum_{r\geq k} \sum_{l=0}^{k-1} \hat{K}(r-l) A_l \leq 1
 \end{equation}
then $F(\beta,h)=0$.
\end{lem}
\begin{proof}[Proof of Lemma \ref{lemma_rho}]
 If (\ref{cdt}) is true then from (\ref{ren}) we can show by induction that for every $l$, $A_l \leq \max\{A_0,\ldots, A_{k-1}\}$. Therefore, 
 \begin{align*}
  F(\beta,h) = \lim \frac{1}{N\gamma}\mathbb{E}\log(Z_N)^{\gamma} &\stackrel{\hbox{(Jensen)}}{\leq} \lim \frac{1}{N\gamma}\log\mathbb{E}(Z_N^{\gamma}) \\&\leq \lim \frac{1}{N\gamma}\log\left(\frac{\mathbb{E}(\tilde{Z}_{N+q+1}^{\gamma})}{K(q+1)^{\gamma}}\right)\\&\leq \lim \frac{1}{N\gamma}\log\left(\frac{A_{N+q+1}}{K(q+1)^{\gamma}}\right)\\&=0.
\end{align*}
Note that in the partition function $Z_N$ considered above, the sum in the Hamiltonian should go from $0$ to $N$ (instead of going from $0$ to $N-1$, or $1$ to $N$) but all these definitions lead to the same free energy in the limit. Moreover, to go from the first line to the second line, we restrict $\tilde{Z}_{N+q+1}$ to renewal trajectories that start with a stretch of length $q+1$.
\end{proof} 
Therefore, our task is now to find parameters $h>h_c^a(\beta)$, $\gamma$ and $k$ that meet the requirements of Lemma \ref{lemma_rho}. Suppose now that
\begin{equation}\label{exp_h}
h=h_c^a(\beta)+\Delta
\end{equation}
 with $\Delta$ small but positive. Then we are going to prove:
\begin{lem}\label{hatK}
For all $\beta$, if $\gamma$ is close enough to $1$ and $\Delta>0$ is small enough then there exists a constant $c(\beta)$ such that
 \begin{equation}\label{tailKhat}
 \forall n\geq 1,\quad \hat{K}(n,\beta,h_c^a(\beta)+\Delta,\gamma) \leq c(\beta) L_1(n) n^{-(1+\alpha)\gamma}.
 \end{equation}
 Moreover, there exists $\beta_0>0$ and $\epsilon>0$ such that for all $\beta\in(0,\beta_0)$, $\Delta\in(0,\epsilon)$, $\gamma\in(1-\epsilon,1)$, (\ref{tailKhat}) holds with $c(\beta)$ replaced by $c(\beta_0)$.
\end{lem}

\begin{proof}[Proof of Lemma \ref{hatK}]
 Let  $n>q$. Then, by decomposing $\hat{Z}_{n,\beta,h,\omega}$ according to the last stretch before $n$, we get
 \begin{equation*}
  \hat{Z}_{n,\beta,h,\omega} = \sum_{l = q+1}^n K(l) \check{Z}_{n-l,\beta,h,\omega}e^{\beta\omega_{n-l}+h}
 \end{equation*}
hence
\begin{equation}\label{aux1}
 \hat{K}(n,\beta,h,\gamma) \leq \sum_{l=q+1}^n K(l)^{\gamma} \mathcal{Z}_{n-l,\beta,h}
 \end{equation}
where \begin{equation*}
       \mathcal{Z}_{j,\beta,h} := \mathbb{E}\left( e^{\gamma(\beta\omega_j+h)} \check{Z}^{\gamma}_{j,\beta,h,\omega} \right).
      \end{equation*}
We now look at the rate of decay of the sequence $\mathcal{Z}_n$. As in (\ref{decompZgamma}), one can decompose $\mathcal{Z}_n$ according to the number of renewals before $n$, except that by definition of the $\check{Z}_k$'s, for each $k$ in $\{1,\ldots, n\}$, the sum over $$\{t_i\geq1, 1\leq i \leq k : t_1+\ldots+t_k = n\}$$ is restricted to $$\{1\leq t_i \leq q, 1\leq i \leq k : t_1 +\ldots + t_k = n\}.$$ By using again (\ref{rev6_1}) and (\ref{rev6_2}), we finally get
\begin{equation}\label{Zdecay}
 \mathcal{Z}_{n,\beta,h} \leq C(\beta,\Delta) \sum_{k=1}^n \sum_{\substack{\overline{t}_1,\ldots,\overline{t}_k\\t_1+\ldots+t_k = n\\t_i\leq q, 1\leq i\leq k}}K(t_1)^{\gamma}\ldots K(t_q)^{\gamma} \prod_{i=1}^{n-1} Q_{\beta,h,\gamma}(\overline{t}_{i},\overline{t}_{i+1})
\end{equation}
where $C(\beta,\Delta)$ is a constant coming from boundary effects (which depends on $\beta$ and $\Delta$ because of (\ref{exp_h})), and for all $x$,$y$ in $E^q$,
\begin{equation*}
 Q_{\beta,h,\gamma}(x,y) = \exp\left(\frac{\beta^2 \gamma^2}{2} +h\gamma +\gamma^2 \beta^2 G(y) \right) K(y_q)^{\gamma}\mathbf{1}_{\{x\rightsquigarrow y\}}.
\end{equation*}
From (\ref{Zdecay}), the behaviour of $\mathcal{Z}_n$ is related to the Perron-Frobenius eigenvalue of $Q_{\beta,h,\gamma}$ restricted to the space $\{1,\ldots,q\}^q$. Let us make this last statement clearer. Since by Theorem \ref{annealed} we have $h_c^a(\beta) = -\frac{\beta^2}{2} - \log \lambda(\beta)$, then
\begin{equation*}
Q_{\beta,h_c^a(\beta),\gamma=1} = \lambda(\beta)^{-1}Q_{\beta}^*,
\end{equation*}
the Perron-Frobenius eigenvalue of which equals $1$ by (\ref{lambda}). By strict monotonicity of the Perron-Frobenius eigenvalue with respect to matrix entries (cf. \cite[Theorem 1.1]{Seneta} or \cite[Appendix A.8]{GG_Book}), the Perron-Frobenius eigenvalue of $Q_{\beta,h_c^a(\beta),\gamma=1}$ restricted to $\{1,\ldots,q\}^q$ (obtained by setting $Q_{\beta,h_c^a(\beta),\gamma=1}(x,y)$ off to $0$ whenever one of the $x_i$'s or $y_i$'s equals $\star$) is then strictly smaller than $1$. Now, by invoking continuity of the Perron-Frobenius eigenvalue with respect to parameters, and by choosing $\gamma$ below but close enough to $1$ as well as $\Delta>0$ small enough, we affirm that the Perron-Frobenius eigenvalue of $Q_{\beta,h_c^a(\beta)+\Delta,\gamma}$ restricted to $\{1,\ldots,q\}^q$ is still strictly smaller than $1$. For such $\gamma$ and $\Delta$, and by using again the monotonicity argument, let $\delta = \delta(\Delta,\gamma)$ be the unique positive number such that the Perron-Frobenius eigenvalue of the matrix $Q^{(\delta)}$ (here we omit other parameters for clarity), defined as 
\begin{equation*}
 \left(\exp\left(\delta y_q + \frac{\beta^2}{2}\gamma(\gamma-1)+\gamma^2\beta^2G(y)+\Delta\gamma\right)\frac{K(y_q)^{\gamma}}{\lambda(\beta)^{\gamma}}\mathbf{1}_{\{x\rightsquigarrow y\}} \right)_{x,y\in \{1,\ldots,q\}^q}
\end{equation*}
is equal to $1$. Let $\nu^{(\delta)}$ be a positive right eigenvector of $Q^{(\delta)}$ associated to $1$. Then (\ref{Zdecay}) becomes 
\begin{align*}
 \mathcal{Z}_{n,\beta,h} \leq & C(\beta,\Delta)\max_{x,y}(\nu^{(\delta)}(y)/\nu^{(\delta)}(x))e^{-\delta n}\\&\times \sum_{k=1}^n \sum_{\substack{\overline{t}_1,\ldots,\overline{t}_k\\t_1+\ldots+t_k = n\\t_i\leq q, 1\leq i\leq k}}K(t_1)^{\gamma}\ldots K(t_q)^{\gamma} \prod_{i=1}^{n-1} Q^{(\delta)}(\overline{t}_{i},\overline{t}_{i+1})\frac{\nu^{(\delta)}(\overline{t}_{i+1})}{\nu^{(\delta)}(\overline{t}_{i})}.
\end{align*}
Since $(x,y)\mapsto Q^{(\delta)}(x,y)\frac{\nu^{(\delta)}(y)}{\nu^{(\delta)}(x)}$ is a Markov chain kernel, we have
\begin{equation}\label{aux2}
 \mathcal{Z}_{n,\beta,h_c^a(\beta)+\Delta,\gamma} \leq c e^{-\delta n}
\end{equation}
where $c$ is a positive constant possibly depending on $\beta$, $\Delta$, and $\gamma$.

From (\ref{defK}), (\ref{aux1}) and (\ref{aux2}) one can then deduce the first point of the lemma. Indeed, one can write
\begin{equation*}
\sum_{l=q+1}^n K(l)^{\gamma}e^{-\delta(n-l)} = \sum_{l=q+1}^{n/2} K(l)^{\gamma}e^{-\delta(n-l)} + \sum_{l=n/2+1}^n K(l)^{\gamma}e^{-\delta(n-l)}.
\end{equation*}
The first sum is simply bounded by $e^{-\delta n/2}\sum_{k\geq q+1}L(k)^{\gamma}k^{-(1+\alpha)\gamma}$ (finite for $\gamma$ close enough to $1$). As for the second sum, we write
\begin{align*}
\sum_{l=n/2+1}^n K(l)^{\gamma}e^{-\delta(n-l)} &= \frac{L(n)^{\gamma}}{n^{(1+\alpha)\gamma}} \sum_{l=n/2+1}^n \left( \frac{n}{l}\right)^{(1+\alpha)\gamma} \left( \frac{L(n)}{L(l)}\right)^{\gamma}e^{-\delta(n-l)}\\
& \leq \frac{2^{(1+\alpha)\gamma}}{1-e^{-\delta}}\sup_{1/2\leq t \leq 1} \left(\frac{L(n)}{L(nt)}\right)^{\gamma},
\end{align*}
which is bounded because $\sup_{1/2\leq t \leq 1} \left(\frac{L(n)}{L(nt)}\right)^{\gamma}$ converges to $1$ by the property of regularly varying functions (cf. \cite[(A.18)]{GG_Book} or \cite{MR1015093}).

We now deal with the second point (uniform version) of the lemma. Let $\overline{\mathbf{1}}$ be the vector such that for all $x$ in $E^q$, if one of the $x_i$'s equals $\star$, then $\overline{\mathbf{1}}(x) = 0$, else $\overline{\mathbf{1}}(x) = 1$. Notice that for all $\beta\geq0$ and $\gamma\leq 1$,
\begin{equation*}
(Q_{\beta,h_c^a(\beta),\gamma}\overline{\mathbf{1}})(x) \leq \exp(\beta^2(1/2+\max G))\lambda(\beta)^{-\gamma}\left(K(1)^{\gamma}+\ldots + K(q)^{\gamma} \right).
\end{equation*}
As a consequence, for every $\eta>0$, there exists $\beta_0>0$ and $\epsilon>0$ such that for all $\beta\leq\beta_0$ and $\gamma$ in $(1-\epsilon,1)$, the Perron-Frobenius eigenvalue of $Q_{\beta,h_c^a(\beta),\gamma}$ restricted to $\{1,\ldots,q\}^q$ is smaller than
$$K(1) +\ldots + K(q) + \eta = P(T_1\leq q) + \eta,$$
which is smaller than $1$ when $\eta$ is small enough. One can then choose $$\Delta<\Delta_0 := -\log(P(T_1\leq q)+\eta)$$ in the lines above to prove the uniform version of the lemma.
\end{proof}

Let now $\Delta$ be close enough to 0 and $\gamma$ close enough to 1 so that the tail behaviour of $\hat{K}$ is as in Lemma \ref{hatK} and $(1+\alpha)\gamma -1 > 1$ (which is possible since $\alpha>1$). Then we have
\begin{align}\label{sum}
 \varrho := \varrho(\beta,h,\gamma,k) = \sum_{r\geq k} \sum_{l=0}^{k-1} \hat{K}(r-l) A_l &\leq c(\beta) \sum_{r\geq k} \sum_{l=0}^{k-1} (r-l)^{-(1+\alpha)\gamma}L_2(r-l)A_l \nonumber\\ &\leq C(\beta) \sum_{l=0}^{k-1}\frac{L_2(k-l)A_l}{(k-l)^{(1+\alpha)\gamma-1}}
\end{align}
where $C(\beta)$ is a constant which can be made uniform on $(0,\beta_0)$ for all $\beta_0$. Our goal in the next sections is to make this last sum small enough, by suitably choosing the shift $\Delta := h_c(\beta) - h_c^a(\beta)$ as a function of $\beta$, and the parameter $k$ as a function of $\Delta$. Theorems \ref{relevance3} and \ref{relevance2} will follow by application of Lemma \ref{lemma_rho}.

\subsection{Proof of Theorem \ref{relevance2}}

In what follows, we set $k=k(\beta)=\frac{1}{a\beta^2}$ and $\Delta = a\beta^2$, with $a$ to be specified later. Then
\begin{equation*}
 \varrho \leq c(\beta) (S_1 + S_2)
\end{equation*}
where
\begin{equation*}
 S_1 = \sum_{l=0}^{k(\beta)-R-1} \frac{L_2(k-l)A_l}{(k-l)^{(1+\alpha)\gamma-1}}
\end{equation*}
and
\begin{equation*}
 S_2 = \sum_{l = k(\beta)-R}^{k(\beta)-1} \frac{L_2(k-l)A_l}{(k-l)^{(1+\alpha)\gamma-1}}
\end{equation*}
with $R\leq k(\beta)$ to be specified. On one hand we have:
\begin{lem}
 $S_1$ can be made small by taking $R$ large enough and $a$ small enough.
\end{lem}
\begin{proof}
 We have
 \begin{equation*}
  A_l = \mathbb{E}(\tilde{Z}_l^{\gamma}) \stackrel{\hbox{(Jensen)}}{\leq} (\mathbb{E}\tilde{Z}_l )^{\gamma} \leq c(\beta) \exp(\gamma F^a(\beta,h_c^a(\beta)+a\beta^2)l)\leq c(\beta)e^{a\beta^2 l},
 \end{equation*}
 which is lower than a constant $c(\beta)$ whenever $l\leq k(\beta)$. To obtain the second inequality we first use that $$\tilde{Z}_l \leq E\left(\exp(H_l(\beta,h,\omega)) \delta_l\right)$$ and then that $$\mathbb{E}E\left(\exp(H_l(\beta,h,\omega)) \delta_l\right) \leq c(\beta) \exp(F^a(\beta,h)l)$$ by superadditivity arguments (cf. \cite[Proof of Theorem 3.1]{Poisat_frc}). To obtain the third inequality we use that $$H_l(\beta, h_c^a(\beta)+a\beta^2) \leq H_l(\beta, h_c^a(\beta)) + a\beta^2 l,$$ which implies $$F^a(\beta,h_c^a(\beta)+a\beta^2)\leq F^a(\beta,h_c^a(\beta))+ a\beta^2 = a\beta^2.$$ Therefore, by summing on $l$ we get
 \begin{equation*}
  S_1 \leq \frac{c(\beta)L_3(R)}{R^{(1+\alpha)\gamma - 2}}
 \end{equation*}
which can be made small by choosing $R$ large enough. Since $R\leq k(\beta)=\frac{1}{a\beta^2}$, this may require $a$ small enough. 
\end{proof}

On the other hand we have $S_2 \leq C_2 \max_{k(\beta)-R\leq l < k(\beta)} A_l$. We will show that this can be made small by taking $a$ small enough, by using the same change of measure argument used in the case of i.i.d. disorder. For this purpose, define
\begin{equation*}
  \frac{d\mathbb{P}_{N,\lambda}}{d\mathbb{P}}(\omega) = \frac{e^{-\lambda \sum_{i=1}^N \omega_i}}{\mathbb{E}(e^{-\lambda \sum_{i=1}^N \omega_i})}.
\end{equation*}
Note that from the Gaussian assumption on $\omega$, this fraction equals $$\exp\left(-\lambda\sum_{i=1}^N\omega_i -\frac{\lambda^2}{2}v_N\right),$$ where $v_N := \var(\sum_{i=1}^N \omega_i)$.
\begin{lem}\label{holder2}
 There exists $c>0$ such that for all $N$, all $\lambda$ and $\gamma$ in $(0,1)$
 \begin{equation*}
  \mathbb{E}(\tilde{Z}_N^{\gamma}) \leq (\mathbb{E}_{N,\lambda}\tilde{Z}_N)^{\gamma} \exp\left(c\frac{\gamma}{1-\gamma}\lambda^2 N\right).
 \end{equation*}
\end{lem}

\begin{proof}
By H\"older inequality we have
\begin{equation}\label{holder}
\mathbb{E}(\tilde{Z}_N^{\gamma}) =  \mathbb{E}_{N,\lambda}\left(\tilde{Z}_N^{\gamma} \frac{d\mathbb{P}}{d\mathbb{P}_{N,\lambda}}(\cdot)\right) \leq (\mathbb{E}_{N,\lambda}\tilde{Z}_N)^{\gamma}  \mathbb{E}_{N,\lambda}\left(\left( \frac{d\mathbb{P}}{d\mathbb{P}_{N,\lambda}}(\cdot)\right)^{1/(1-\gamma)} \right)^{1-\gamma}
\end{equation}
The last factor on the right-hand side of (\ref{holder}) is equal to
\begin{equation*}
 \mathbb{E}\left[\left( e^{\lambda\sum\omega_i + \frac{\lambda^2}{2}v_N} \right)^{\frac{1}{1-\gamma}} e^{-\lambda\sum\omega_i -\frac{\lambda^2}{2}v_N}\right]^{1-\gamma} = e^{\frac{\lambda^2\gamma}{2(1-\gamma)}v_N}.
\end{equation*}
and the lemma is true with $c := \sup_{N\geq1} (v_N/N)$, which is finite since $v_N \sim N(1 + 2\sum_{k=1}^q \rho_k)$ as $N$ tends to $+\infty$.
\end{proof}

If $N=j$ and $\lambda = \frac{1}{\sqrt{j}}$, we get:
\begin{equation}\label{Aj}
 A_j \leq (\mathbb{E}_{j,1/\sqrt{j}}\tilde{Z}_j)^{\gamma} \exp\left(c\frac{\gamma}{1-\gamma}\right).
\end{equation}

\begin{pr}
 If $h=h_c^a(\beta)+\Delta$ then
 \begin{equation*}
  \mathbb{E}_{j,\lambda}(\tilde{Z}_j) \leq c(\beta) E_{\beta}\left( e^{(\Delta-\overline{\rho}\beta\lambda) \sum_{i=1}^j \delta_i} \right)
 \end{equation*}
where $\overline{\rho} = 1 + 2\sum_{k=1}^q \rho_k$.
\end{pr}
\begin{proof}
Computations give:
\begin{align*}
 \mathbb{E}_{j,\lambda}(\tilde{Z}_j) &= E\mathbb{E}_{j,\lambda}\left(e^{\sum_{k=0}^{j-1}(\beta\omega_k+h)\delta_k}\mathbf{1}_{\{j\in\hat{\tau}\}} \right)\\
 & \leq E\left(\mathbb{E}\left(e^{\sum_{k=0}^{j-1}(\beta\omega_k + h)\delta_k - \lambda\sum_{k=0}^{j-1}\omega_k} \right) e^{-\frac{\lambda^2}{2}v_j}\right)\\
 & = E\left( e^{h\sum\delta_k +\frac{1}{2}\var(\sum_{k=0}^{j-1}\omega_k(\beta\delta_k - \lambda))}\right)e^{-\frac{\lambda^2}{2}v_j},
\end{align*}
and
\begin{align*}
&\var(\sum_{k=0}^{j-1}\omega_k(\beta\delta_k - \lambda)) \\&= \sum_{k=0}^{j-1} (\beta\delta_k -\lambda)^2 + 2 \sum_{0\leq m<n\leq j-1}(\beta\delta_m -\lambda)(\beta\delta_n-\lambda)\rho_{n-m}\\&= \lambda^2j + 2\lambda^2\sum_{0\leq m<n\leq j-1}\rho_{n-m} +\beta^2 \sum_{k=0}^{j-1}\delta_k - 2\beta\lambda\sum_{k=0}^{j-1}\delta_k\\& + 2\beta^2 \sum_{0\leq m<n\leq j-1} \delta_m\delta_n\rho_{n-m} -2\beta\lambda\sum_{0\leq m<n\leq j-1}\delta_n \rho_{n-m}\\& -2\beta\lambda \sum_{0\leq m <n\leq j-1}\delta_m \rho_{n-m}
\end{align*}
and in the last equality, the sum of the first two terms equals $\lambda^2 v_j$. Hence, at $h=h_c^a(\beta) +\Delta$:
\begin{align}
 \mathbb{E}_{j,\lambda}(\tilde{Z}_j) &\leq C_1(\beta) E\left( e^{(\Delta - \overline{\rho}\beta\lambda -\log\lambda(\beta))\sum\delta_n +\beta^2 \sum\delta_n\delta_m \rho_{n-m}}  \right)\nonumber\\
 &\leq C_2(\beta) E_{\beta}\left(e^{(\Delta -\overline{\rho}\beta\lambda)\sum_{k=1}^j\delta_k} \right).\label{E_j_lambda}
\end{align}
where $C_1(\beta)$ and $C_2(\beta)$ are constants which are uniform on $\beta\leq\beta_0$, for all $\beta_0$ (remark that due to boundary effects, $\lambda$ should also appear in $C_1(\beta)$, but this is harmless since we will choose $|\lambda| = 1/\sqrt{j} \leq 1$).
\end{proof}

If $\Delta = a\beta^2$, and $a$ small enough, then for $j\leq k(\beta)=\frac{1}{a\beta^2}$,
\begin{equation*}
 \Delta - \frac{\overline{\rho}\beta}{\sqrt{j}} \leq -\frac{c_1}{2k(\beta)\sqrt{a}},
\end{equation*}
(the constant is uniform in $\beta$) hence
\begin{equation*}
 \max_{k(\beta)-R\leq j < k(\beta)} \mathbb{E}_{j,1/\sqrt{j}}(\tilde{Z}_j) \leq e^{c_1\sqrt{a}\beta^2\frac{R}{2}} E_{\beta}\left(\exp\left(-\frac{c_1}{2\sqrt{a}k(\beta)}\arrowvert \tau \cap \{1,\ldots,k(\beta)\} \arrowvert \right) \right).
\end{equation*}
We have used the inequality $\imath_j \geq \imath_{k(\beta)}-R$ for the range of $j$'s appearing in the maximum. We can make the last term as small as we want by taking $a$ small enough, which proves the second point of the theorem. For the first point, we need to prove that the procedure is uniform in $\beta\leq\beta_0$. Indeed, we shall prove:

\begin{lem}\label{lem_lim2}
 \begin{equation}\label{lim}
  \lim_{c\rightarrow\infty} \limsup_{\beta\rightarrow 0} E_{\beta}\left(e^{-\frac{c}{k(\beta)}|\tau\cap\{1,\ldots,k(\beta)\}|} \right) = 0.
 \end{equation}
\end{lem}

\begin{proof}
 This is a bit trickier than in the i.i.d case because also the law of $\tau$ depends on $\beta$. First, let us remark that there exists a coupling (of the modulating Markov chains with kernel $\tilde{Q}^*_{\beta}$) such that the expectation in (\ref{lim}) can be written
\begin{equation*}
 E\left(\exp\left(-\frac{c}{k(\beta)}\arrowvert \tau_{\beta} \cap \{1,\ldots, k(\beta)\} \arrowvert \right) \right).
\end{equation*}
Since $\tau_{\beta}$ converges to $\tau$ and $k(\beta)$ tends to $+\infty$ as $\beta$ goes to $0$, and 
\begin{equation*}
 \frac{\arrowvert \tau \cap \{1,\ldots,N\} \arrowvert}{N} \stackrel{\hbox{a.s}}{\rightarrow} \frac{1}{m} := \frac{1}{\sum_{n\geq1} nK(n)},
\end{equation*}
we expect that  the random variable $\frac{\arrowvert \tau_{\beta}\cap \{1,\ldots,k(\beta)\}\arrowvert}{k(\beta)}$
converges to $1/m$, but the result is not clear because there is a problem of uniformity in $\beta$. However, we can prove by hand that the convergence holds in law. We can show for example convergence of the cumulative distribution function. Since \begin{align*}P\left(\frac{|\tau_{\beta}\cap \{1,\ldots,k(\beta)\}|}{k(\beta)}\geq x\right) &= P(|\tau_{\beta}\cap\{1,\ldots,k(\beta)\}|\geq \lceil xk(\beta)\rceil)\\& = P(\tau_{\beta,\lceil xk(\beta)\rceil}\leq k(\beta))\\ & = P(\tau_{\beta,\lceil xk(\beta)\rceil}/k(\beta) \leq 1),\end{align*} it is enough to show that $\tau_{\beta,\lceil xk(\beta)\rceil}/k(\beta)$ converges in law to $mx$ as $\beta$ tends to $0$. We will prove this point by means of convergence of the Laplace transforms. From now on, we assume $xk(\beta)$ is an integer to avoid repeated use of $\lceil \cdot \rceil$. First we define $\Phi_{\beta}$ a matrix of Laplace transforms. For all $\beta\geq 0$, $\lambda\geq 0$, $x$ and $y$ in $E^q$, $\Phi_{\beta,x,y}(\lambda) := \varphi_{y_q}(\lambda) \tilde{Q}^*_{\beta}(x,y)$ where $\tilde{Q}^*_{\beta}$ is the  transition matrix defined in (\ref{tildeQstar}) and the $\varphi_t$'s are the following Laplace transforms:
\begin{equation*}
\varphi_t(\lambda) = \left\{ \begin{array}{ccc} e^{-\lambda t} & \hbox{ if } & 1\leq t \leq q\\ \sum_{t> q} e^{-\lambda t} \frac{K(t)}{K(\star)} & \hbox{ if }& t = \star. \end{array} \right.
\end{equation*} 
Then
\begin{equation*}
E(e^{-\lambda \frac{\tau_{\beta, xk(\beta)}}{k(\beta)}}) = \mu_0 \Phi_{\beta}^{xk(\beta)}(\frac{\lambda}{k(\beta)})\mathbf{1}
\end{equation*}
where $\mathbf{1}$ is the column vector with all coordinates equal to $1$ and $\mu_0$ is the initial law of the modulating Markov chain. Define also
\begin{equation*}
m(t) = \left\{ \begin{array}{ccc} t & \hbox{ if } & 1\leq t \leq q\\ \sum_{t>q} t\frac{K(t)}{K(\star)} \end{array} \right.
\end{equation*}
Then $\varphi_t(\lambda) = 1 -\lambda m(t)(1+o_{\lambda}(1))$ and $\tilde{Q}^*_{\beta} = \tilde{Q}^*_0 + A \beta^2 (1+o_{\beta}(1))$, with
\begin{equation}\label{A}A\mathbf{1}=\mathbf{0},\end{equation} so there exists a matrix $\epsilon_{\beta}(\lambda) = o_{\beta}(1)$ for all $\lambda\geq 0$ so that \begin{equation}\label{phi} \Phi_{\beta}(\lambda a\beta^2) =       \tilde{Q}^*_0 + \beta^2 (A -\lambda a M) + \beta^2 \epsilon_{\beta}(\lambda),\end{equation} where $M(x,y) = m(y_q)\tilde{Q}^*_0(x,y)$ and \begin{equation}\label{M} M\mathbf{1} = m\mathbf{1}.\end{equation} Notice that from (\ref{A}) and (\ref{M}) we have for all $k\geq 0$,
\begin{equation}\label{aux3}
(\tilde{Q}^*_0 + \beta^2(A - a\lambda M))^k \mathbf{1} = (1 - a\lambda m \beta^2)^k\mathbf{1}
\end{equation}
and if we choose $k = k(\beta) = \frac{1}{a\beta^2}$ and make $\beta$ tend to $0$, the right-hand side of (\ref{aux3}) converges to $e^{-\lambda mx}$, which is the limit we want to obtain. It remains to control the remainder term. Let
\begin{equation*}
R_{\beta,n}(\lambda) = \mu_0\left( \phi_{\beta}^n(\lambda a\beta^2) - (\tilde{Q}^*_0 + \beta^2(A - a\lambda M))^n \right)\mathbf{1}.
\end{equation*}
From (\ref{phi}), for all $\lambda\geq 0$ there exists $c>0$ such that
\begin{align*}
|R_{\beta,n}(\lambda)| &\leq \sum_{k=1}^n C^k_n (1+c\beta^2)^{n-k}(\beta^2 \max_{x,y\in E^q}|\epsilon_{\beta}(x,y)|)^k\\
& = (1 + c\beta^2 + \beta^2\Vert \epsilon_{\beta} \Vert)^n - (1+c\beta^2)^n
\end{align*}
and if we set $n = x/(a\beta^2)$, the two terms will tend to the same quantity as $\beta$ tends to $0$.
\end{proof}

We make a brief summary of the proof in the case $\alpha>1$. Uniformly in $\beta\leq\beta_0$ (for any $\beta_0$): set $h=h_c^a(\beta)+a\beta^2$ and choose $a$ small and $\gamma$ close to one so that $(1+\alpha)\gamma - 1> 1$ and Lemma \ref{hatK} holds. Again, if necessary, take $a$ even smaller so that $S_2$ is small and $R$ large enough to make $S_1$ small. All in all, $\varrho$ is smaller than $1$ so with Lemma \ref{lemma_rho} we can conclude that $F(\beta,h_c^a(\beta) + a\beta^2)=0$.

\subsection{Proof of Theorem \ref{relevance3}}
Let us fix $\beta>0$ and set
\begin{equation*}
 k = k(\Delta) := F^a(\beta,h_c^a(\beta)+\Delta)^{-1} 
 \end{equation*}
 so that from Proposition \ref{ann_expo},
 \begin{equation*}
k(\Delta) \stackrel{\Delta\searrow 0}{\sim} (L'_{\beta}(1/\Delta))^{-1}\Delta^{-1/\alpha}.
\end{equation*}
Using (\ref{E_j_lambda}) with $\lambda= 1/\sqrt{j}$ and
\begin{equation*}
 0\leq\Delta\sqrt{j} \leq \Delta\sqrt{k} \stackrel{\Delta\searrow 0}{\sim} (L'_{\beta}(1/\Delta))^{-1/2} \Delta^{1-\frac{1}{2\alpha}}\stackrel{(\alpha>1/2)}{\longrightarrow} 0,
\end{equation*}
there exists positive constants $c_1$ and $c_2$ (that may depend on $\beta$) such that for all $j\in\{1,\ldots,k\}$,
\begin{equation*}
 \mathbb{E}_{j,1/\sqrt{j}}(\tilde{Z}_j) \leq c_1 E_{\beta}\left(e^{-\frac{c_2}{\sqrt{j}}|\tau\cap\{1,\ldots,j\}|}\delta_j\right).
\end{equation*}
For all $x$ in $E^q$, let us denote by $\tau^{(x)}$ the renewal process defined by
\begin{equation*}
 \tau^{(x)} = \{\tau_n, n\geq q : \overline{T}^*_{n-q} = x\}.
\end{equation*}
where we remind that $\overline{T}^*_n = (T^*_n,\ldots, T^*_{n+q-1})$. In \cite[Relation (21)]{Poisat_frc} it was proved that when $(\tau_n)_{n\geq0}$ has law $P_{\beta}$, the interarrival kernel of $\tau^{(x)}$, that will be denoted by $(K_{\beta}^{(x)}(n))_{n\geq 1}$, satisfies
\begin{equation}\label{Kx}
 K_{\beta}^{(x)}(n) \stackrel{n\rightarrow +\infty}{\sim} c(\beta,x) K(n)
\end{equation}
for some positive constant $c(\beta,x)$. Then
\begin{align*}
 E_{\beta}\left(e^{-\frac{c_2}{\sqrt{j}}|\tau\cap\{1,\ldots,j\}|}\delta_j\right) &= \sum_{x\in E^q} E_{\beta}\left(e^{-\frac{c_2}{\sqrt{j}}|\tau\cap\{1,\ldots,j\}|}\delta_j^{(x)}\right)\\
 &\leq \sum_{x\in E^q} E_{\beta}\left(e^{-\frac{c_2}{\sqrt{j}}|\tau^{(x)}\cap\{1,\ldots,j\}|}\delta_j^{(x)}\right). 
\end{align*}
For every term in the last sum, we use (\ref{Kx}) and Proposition A.2 of \cite{Derrida_al_relevance} to get
\begin{equation*}
 E_{\beta}\left(e^{-\frac{c_2}{\sqrt{j}}|\tau^{(x)}\cap\{1,\ldots,j\}|}\delta_j^{(x)}\right) \leq C(\beta,x) \frac{L(j)}{j^{\alpha}}.
\end{equation*}
Therefore, from (\ref{Aj}) there exists a constant $c$ (depending on $\beta$, which is fixed here) so that:
\begin{equation}\label{Aj2}
 A_j \leq c(L(j))^{\gamma}j^{-\gamma\alpha},
\end{equation}
where $(L(\cdot))^{\gamma}$ is a slowly varying function. In the following we allow the value of $c$ to change from line to line. We split the sum in (\ref{sum}) in two parts. First, one gets using properties of slowly varying functions:
\begin{align*}
 \sum_{j=1}^{k/2} \frac{L(k-j)A_j}{(k-j)^{(1+\alpha)\gamma-1}} &\leq c\frac{L(k/2)}{k^{(1+\alpha)\gamma-1}} \sum_{j=1}^{k/2} \frac{(L(j))^{\gamma}}{j^{\alpha\gamma}} \\&\leq \frac{c(L(k/2))^{(1+\gamma)}}{k^{\gamma(2\alpha+1)-2}},
\end{align*}
which goes to $0$ as $k\rightarrow +\infty$ (i.e $\Delta\rightarrow 0$) by choosing $\gamma$ close enough to $1$, since
$$\gamma(2\alpha+1)-2 \stackrel{\gamma\rightarrow 1}{\longrightarrow} 2\alpha-1 >0.$$
Next, we have
\begin{align*}
 \sum_{j=k/2+1}^{k-1} \frac{L(k-j)A_j}{(k-j)^{(1+\alpha)\gamma-1}} &\leq c\frac{(L(k/2))^{\gamma}}{k^{\gamma\alpha}}\sum_{j=1}^{k/2} \frac{L(j)}{j^{(1+\alpha)\gamma-1}}\\
 &\leq \frac{c(L(k/2))^{(1+\gamma)}}{k^{\gamma(2\alpha+1)-2}},
\end{align*}
which goes to $0$ as $k\rightarrow +\infty$ for the same reason as above. We conclude from Lemma \ref{lemma_rho} that for all $\beta>0$, for $\Delta>0$ small enough, $F(\beta, h_c^a(\beta)+\Delta)=0$ and then $h_c(\beta)>h_c^a(\beta)$.

\begin{rmk}\label{rmk2}
 To obtain the bound (\ref{shift_h_c}) as in the i.i.d. case, one should set $\Delta = a\beta^{\frac{2\alpha}{2\alpha-1}+\epsilon}$ ($\epsilon>0$ arbitrarily small) and get bounds on $(A_j)_{1\leq j \leq k(\beta)}$ that are uniform in $\beta$ on $(0,\beta_0)$ (for some $\beta_0>0$), but this would require a finer analysis than the one we provide. When $\alpha>1$, this was made possible by Lemma \ref{lem_lim2}, the proof of which uses the fact that the terms in $\tilde{Q}_{\beta}^*$ and $\Delta (= a\beta^2)$ are of the same order in $\beta$, which is no longer true if we choose $\Delta = a\beta^{\frac{2\alpha}{2\alpha-1}+\epsilon}$, as it should be when $1/2 < \alpha < 1$.
 \end{rmk}



\bibliographystyle{model1b-num-names}
\bibliography{references}







\end{document}